\documentclass[12pt, twoside, leqno]{article}

\usepackage{lmodern}
\usepackage{amsmath,amsthm,amscd}
\usepackage{amsfonts}
\usepackage{amssymb,latexsym}
\usepackage{enumerate}
\usepackage[T1]{fontenc}
\usepackage{MnSymbol}

\usepackage{geometry}
\usepackage{marginnote}

\usepackage{graphicx}
\usepackage{multirow}
\usepackage{caption}
\usepackage{subfig}
\usepackage{hyperref}
\usepackage{pgfplots}

\pgfplotsset{compat=1.16}

\DeclareMathOperator{\lehmer}{lc}
\DeclareMathOperator{\numbering}{n}
\DeclareMathOperator{\lex}{lex}

\theoremstyle{definition}
\newtheorem{thm}{Theorem}[section]
\newtheorem{cor}[thm]{Corollary}
\newtheorem{lem}[thm]{Lemma}

\newtheorem{defin}[thm]{Definition}
\newtheorem{rem}[thm]{Remark}

\numberwithin{equation}{section}

\renewenvironment{proof}[1]{\noindent  \ignorespaces \textit{Proof#1.}}{\hfill\qedsymbol \vspace{0.5cm}}

\begin{document}

\baselineskip=17pt

\title{The Lehmer factorial norm on $S_{n}$}

\author{Pawe\l{} Zawi\'slak\\
Department of Mathematics and Mathematical Economics \\
SGH Warsaw School of Economics \\
Al. Niepodleg\l{}o\'sci 162, 02-554 Warszawa, Poland \\
E-mail: pzawis@sgh.waw.pl \\
}

\date{\today}

\maketitle

\renewcommand{\thefootnote}{}

\renewcommand{\thefootnote}{\arabic{footnote}}
\setcounter{footnote}{0}

\begin{abstract}
We introduce a new family of norms on the permutation groups $S_{n}$. We examine their properties.
\end{abstract}

\section{Introduction}\label{sec:Introduction}
Metrics on the permutation groups $S_{n}$ were considered in many different contexts. 

On one side, permutations can be used used as rankings, therefore some metrics on permutations originate from attempts of comparing rankings. Many well known measures of similarity between rankings lead to definition of metric on $S_{n}$. The most popular measures of similarity are Kendall's $\tau$ (\cite{K}) and Spearman's $\rho$ (\cite{S}). These two measures leads to Kendall's distance and Spearman's distance. Together with Spearman footrule (also known as Manhattan distance) and Hamming distance, these four metrics are the most popular metrics on $S_{n}$ used in the statistics (\cite{DG} and \cite{DH}).

The generalisations of these metrics in many different contexts we considered for example in \cite{KW}, \cite{LH}, \cite{LZH}, \cite{KV}, \cite{QDRL}, \cite{LY1}, \cite{LY2}, \cite{WSSC}, \cite{FSS} and \cite{PP}. 

On the other side, the study of statistical properties of natural valued functions on $S_{n}$ has almost two hundred years of history, started with \cite{R} and \cite{M}, and continued by many authors (see for example \cite{FZ}, \cite{SS} and \cite{CSZ}).

On the third side, natural inclusions of $S_{k}$ in $S_{l}$ (for $k<l$) lead to a limit object $S_{\infty}$. Metrics on some limit objects related to groups were explored for example in \cite{C}, \cite{TW}, \cite{W} and \cite{GMZ}.  

In this paper we present a slightly different approach, which can be considered as transversal to the previous three. The demand for the metric satisfying conditions (\ref{thm:LehmerNormMinimal})-(\ref{thm:LehmerNormTransposition}) of Theorem \ref{thm:LehmerNorm} comes from the analysis of the votings' networks. Presicely, all most popular metrics on $S_{n}$ do not differentiate between the change on first two positions of the ranking and the change on last two positions. The metric coming from the norm presented in this paper do -- see Theorem \ref{thm:LehmerNorm} (\ref{thm:LehmerNormOrder}). Additionally, Theorems \ref{thm:Distribution} and \ref{thm:Recursion} describe  distributions of the new norm (we call it the Lehmer norm) on all permutation groups $S_{n}$ as well as its distribution on $S_{\infty}$.

We do not consider many others research contexts of permutations metrics. For more of these contexts see for example \cite{DH}.

This paper is organised as follows.  Section \ref{sec:Notation} contains the basic definitions and notation. In Section \ref{sec:LehmerNorm} we define the main object of this article - the Lehmer factorial norm. The definition bases on the notion on the Lehmer code. The properties of the Lehmer code are described in Lemmas \ref{lem:LehmerCodeProperties} and \ref{lem:LehmerCodeTransposition} as well as in Corollary \ref{cor:LehmerCodeProperties}. Theorem \ref{thm:LehmerNorm} contains the basic attributes of the Lehmer norm. In Section \ref{sec:Distribution} we focus on the distribution of the Lehmer norm. This distribution is fully described in Theorems \ref{thm:Distribution} and \ref{thm:Recursion} together with Lemma \ref{lem:DecompositionNumbers}.

\section{Basic definitions and notation}\label{sec:Notation}
In this section we recall some basic definitions used in this paper as well as we set some notation.

In this article $\mathbb{N}$ denotes the set of all natural numbers, starting at $0$, whereas $\mathbb{N}_{+}$ -- the set of all positive natural numbers. For $n \in \mathbb{N}_{+}$ by $[n]$ we denote the set $\{1,2,\ldots,n\}$ and by $S_{n}$ -- the group of all permutations of $[n]$. $S_{\infty}$ stands for the group of all permutations of $\mathbb{N}_{+}$ with a finite support.

A permutation $\sigma \in S_{n}$ is denoted by 
\begin{displaymath}
\sigma=(\sigma(1),\sigma(2),\ldots,\sigma(n))
\end{displaymath} 
In particular $e_{n}=(1,2,\ldots,n)$ denotes the identity permutation. 

By $\sigma^{-1}$ we denote the inverse permutation to $\sigma$, by $\sigma\tau$ -- the composition of $\sigma$ and $\tau$, defined by $(\sigma\tau)(i)=\sigma(\tau(i))$ for $i=1,2,\ldots,n$, whereas $\bar{\sigma}$ stands for the permutation reverse to $\sigma$, given by $\bar{\sigma}(i)=\sigma(n+1-i)$ for $i=1,2,\ldots,n$.

For $s=1,2,\ldots,n-1$ let 
\begin{displaymath}
\sigma_{s}=(1,2,\ldots,s-1,s+1,s,s+2,\ldots,n)
\end{displaymath}
(so $\sigma_{s}$ is the adjacent transposition -- $(s,s+1)$ in the cycle notation).

\begin{defin}\label{def:LehmerCode}
For a permutation $\sigma \in S_{n}$ its \emph{Lehmer code} $\lehmer(\sigma)$ (see \cite{G}) is defined by
\begin{displaymath}
\lehmer(\sigma)=[c_{1}(\sigma),c_{2}(\sigma),\ldots,c_{n}(\sigma)]
\end{displaymath}
where the numbers $c_{i}(\sigma)$ (for $i=1,2,\ldots,n$) are given by
\begin{displaymath}
c_{i}(\sigma)=\left|\{j \in [n]: j>i \textrm{ and } \sigma(j)<\sigma(i)\}\right|
\end{displaymath}
\end{defin}

The Lehmer code of $\sigma$ coincides with the factorial number system representation of its position in the list of permutations of $[n]$ in the lexicographical order (numbering the positions starting from $0$) -- compare \cite{G} to \cite{L1} and \cite{L2}.

The Lehmer codes of permutations $\sigma \in S_{3}$ are presented in Table \ref{table:LehmerNorm}.

\section{The Lehmer factorial norm}\label{sec:LehmerNorm}

In this section we define the Lehmer factorial norm on the group $S_{n}$. We also examine its basic features.

We start with establishing some basic properties of the Lehmer code. To do this, we need some technical notation. For a permutation $\sigma \in S_{n}$ and for $i=1,2,\ldots,n$ let 
\begin{displaymath}
C(\sigma)_{i}=\{j \in [n]: j>i \textrm{ and } \sigma(j)<\sigma(i)\} \textrm{ and } 
A(\sigma)_{i}=[i]\cup C(\sigma)_{i} 
\end{displaymath}
Note, that if we denote the cardinality of $X$ by $|X|$, then $|C(\sigma)_{i}|=c_{i}(\sigma)$ and $|A(\sigma)_{i}|=i+c_{i}(\sigma)$.

\begin{lem}\label{lem:LehmerCodeProperties} 
For all permutations $\sigma, \tau \in S_{n}$ and for all $i=1,2,\ldots,n$ the following hold:
\begin{enumerate}[(i)]
\item\label{lem:LehmerCodePropertiesValue} $\sigma(i) \leq i + c_{i}(\sigma)$,
\item\label{lem:LehmerCodePropertiesComposition} $c_{i}(\sigma\tau)\leq c_{i}(\tau)+c_{\tau(i)}(\sigma)$,
\item\label{lem:LehmerCodePropertiesInverse} $\sigma^{-1}$ determines the bijection between $A(\sigma^{-1})_{\sigma(i)}$ and $A(\sigma)_{i}$. In particular 
\begin{displaymath}
i+c_{i}(\sigma)=\sigma(i)+c_{\sigma(i)}(\sigma^{-1})
\end{displaymath}
\end{enumerate}
\end{lem}
\begin{proof}{}
(\ref{lem:LehmerCodePropertiesValue}) Note that $|\{j\in [n]:j>i\}|=n-i$, so 
\begin{displaymath}
|\{j\in[n]: j>i \textrm{ and }  \sigma(j)>\sigma(i) \}|=n-i-c_{i}(\sigma)
\end{displaymath}
On the other hand
\begin{displaymath}
|\{\sigma(j) : j \in [n] \textrm{ and } \sigma(j)>\sigma(i) \}|=n-\sigma(i)
\end{displaymath}
And since
\begin{displaymath}
\sigma\Big[\{j \in [n]: j>i \textrm{ and } \sigma(j)>\sigma(i) \}\Big] \subseteq 
\{\sigma(j): j \in [n] \textrm{ and } \sigma(j)>\sigma(i) \} 
\end{displaymath}
it follows that $n-i-c_{i}(\sigma)\leq n-\sigma(i)$.

(\ref{lem:LehmerCodePropertiesComposition}) Choose $k \in C(\sigma\tau)_{i}$. If $\tau(k)<\tau(i)$, then $k\in C(\tau)_{i}$. Otherwise 
$\tau(k)>\tau(i)$ and $\sigma(\tau(k))<\sigma(\tau(i))$, therefore $\tau(k)\in C(\sigma)_{\tau(i)}$.

(\ref{lem:LehmerCodePropertiesInverse})
Choose $k \in A(\sigma^{-1})_{\sigma(i)}$ and let $l=\sigma^{-1}(k)$. We will show that $l \in A(\sigma)_{i}$. There are two possible cases.
\begin{enumerate}[(a)]
\item $k\leq \sigma(i)$: If $l \leq i$, then $l \in [i] \subseteq A(\sigma)_{i}$. Otherwise $l>i$, and since $\sigma(l) = k \leq \sigma (i)$, it follows that $\sigma(l)< \sigma(i)$ hence $l \in C(\sigma)_{i}\subseteq A(\sigma)_{i}$.
\item $k > \sigma(i)$: Therefore $k \in C(\sigma^{-1})_{\sigma(i)}$, so $l=\sigma^{-1}(k) < \sigma^{-1}(\sigma(i))=i$ and thus $l \in [i] \subseteq A(\sigma)_{i}$.
\end{enumerate}
We have shown that $\sigma^{-1}\left[A(\sigma^{-1})_{\sigma(i)}\right]\subseteq A(\sigma)_{i}$. Replacing $\sigma$ with $\sigma^{-1}$ leads to the second inclusion, which completes the proof.
\end{proof}

As an obvious conclusion of Lemma \ref{lem:LehmerCodeProperties} (\ref{lem:LehmerCodePropertiesInverse}) we get:
\begin{cor}\label{cor:LehmerCodeProperties}
Elements of the Lehmer code of the inverse permutation to $\sigma$ are given by
\begin{displaymath}
c_{i}(\sigma^{-1})=c_{\sigma^{-1}(i)}(\sigma)+\sigma^{-1}(i)-i
\end{displaymath}
for $i=1,2,\ldots,n$.
\end{cor}

In the next lemma we describe how the Lehmer code changes when a permutation is multiplied by an adjacent transposition.
\begin{lem}\label{lem:LehmerCodeTransposition}
\begin{enumerate}[(i)]
\item\label{lem:LehmerCodeTranspositionSingle} $c_{i}(\sigma_{s})=\delta_{is}$ (the Kronecker delta) for 
$i=1,2,\ldots,n$ and $s=1,2,\ldots,n-1$.
\item\label{lem:LehmerCodeTranspositionMultiplication} Let $\tau = \sigma\sigma_{s}$. If $\sigma(s)<\sigma(s+1)$, then
\begin{displaymath}
\begin{cases}
c_{i}(\tau)=c_{i}(\sigma) \textrm{ for } i\neq s,s+1 \\
c_{s}(\tau)=c_{s+1}(\sigma)+1 \\
c_{s+1}(\tau)=c_{s}(\sigma)
\end{cases}
\end{displaymath}
otherwise
\begin{displaymath}
\begin{cases}
c_{i}(\tau)=c_{i}(\sigma) \textrm{ for } i\neq s,s+1 \\
c_{s}(\tau)=c_{s+1}(\sigma) \\
c_{s+1}(\tau)=c_{s}(\sigma)-1
\end{cases}
\end{displaymath}
\end{enumerate}
\end{lem}
\begin{proof}{}
(\ref{lem:LehmerCodeTranspositionSingle}) Follows definitions of the Lehmer code and $\sigma_{s}$.

(\ref{lem:LehmerCodeTranspositionMultiplication}) Note first, that 
\begin{displaymath}
\tau=(\sigma(1),\sigma(2),\ldots,\sigma(s-1),\sigma(s+1),\sigma(s),\sigma(s+2),\ldots,\sigma(n))
\end{displaymath}
Threrefore $c_{i}(\sigma)=c_{i}(\tau)$ for $i \neq s,s+1$ (for $i<s$ we have: $s \in C(\sigma)_{i}$ if and only if $s+1 \in C(\tau)_{i}$).

Suppose that $\sigma(s)<\sigma(s+1)$. In this case 
\begin{displaymath}
C(\tau)_{s}=C(\sigma)_{s+1}\cup\{s+1\} \textrm{ and }C(\tau)_{s+1}=C(\sigma)_{s}
\end{displaymath}
If $\sigma(s)>\sigma(s+1)$, then
\begin{displaymath}
C(\tau)_{s}=C(\sigma)_{s+1} \textrm{ and } C(\tau)_{s+1}=C(\sigma)_{s}\setminus \{s+1\}
\end{displaymath}
This finishes the proof.
\end{proof}

Now we are ready to define the Lehmer factorial norm on $S_{n}$.
\begin{defin}\label{def:LehmerNorm}
Let $\sigma \in S_{n}$ be a permutation with the Lehmer code 
\begin{displaymath}
\lehmer(\sigma)=[c_{1}(\sigma),c_{2}(\sigma),\ldots,c_{n}(\sigma)]=[k_{n-1}(\sigma),k_{n-2}(\sigma),\ldots,k_{0}(\sigma)]
\end{displaymath}
(here $k_{i}(\sigma)=c_{n-i}(\sigma)$ for $i=0,1,\ldots,n-1$). 
The \emph{Lehmer factorial norm (with base $2$)} $\mathcal{LF}_{2}:S_{n} \to \mathbb{N}$ is given by 
\begin{displaymath}
\mathcal{LF}_{2}(\sigma)=\sum_{i=0}^{n-1}\left[2^{i}-2^{i-k_{i}(\sigma)}\right]
\end{displaymath}
\end{defin}

\begin{rem}\label{rem:LehmerNorm}
For a number $m \in \mathbb{N}$ let
\begin{displaymath}
m=k_{n-1}\cdot (n-1)!+\ldots+k_{1}\cdot 1! + k_{0}\cdot 0!
\end{displaymath}
be the (unique!) decomposition of $m$ in such a way, that $0\leq k_{i} \leq i!$ for $i=0,1,\ldots,n-1$ (in particular $k_{0}=0$). Therefore $m$ has the following factorial number system representation
\begin{displaymath}
k_{n-1}:\ldots:k_{1}:0_{!}
\end{displaymath}  
Consider the function $LF_{2}:\mathbb{N}\to\mathbb{N}$ given by 
\begin{displaymath}
 LF_{2}(m)=LF_{2}\left(k_{n-1}\cdot (n-1)!+\ldots+k_{1}\cdot 1! + k_{0}\cdot 0!\right)=\sum_{i=0}^{n-1}\left[2^{i}-2^{i-k_{i}}\right]
\end{displaymath}
For $\sigma \in \mathbb{N}$ let $\numbering_{\lex}(\sigma)$ be the position of $\sigma$ in the lexicographical order (numbering starting from $0$). Then
\begin{displaymath}
\mathcal{LF}_{2}(\sigma)=LF_{2}(\numbering_{\lex}(\sigma))
\end{displaymath}
\end{rem}

The values $\mathcal{LF}_{2}(\sigma)$ for $\sigma \in S_{3}$ are presented in Table \ref{table:LehmerNorm}.

The next theorem yields information about the basic properties of the Lehmer norm.

\begin{thm}\label{thm:LehmerNorm} 
The norm $\mathcal{LF}_{2}$ satisfies the following:
\begin{enumerate}[(i)]
\item\label{thm:LehmerNormMinimal} $\mathcal{LF}_{2}(e_{n})=0$ is minimal and $e_{n}$ (the identity) is the only permutation with this property.
\item\label{thm:LehmerNormMaximal} $\mathcal{LF}_{2}(\bar{e_{n}})=2^{n}-(n+1)$ is maximal and $\bar{e_{n}}$ (the reverse of the identity) is the only permutation with this property.
\item\label{thm:LehmerNormOrder} $\mathcal{LF}_{2}(\sigma_{s})=2^{n-1-s}$ (recall that $\sigma_{s}$ denotes the adjacent transposition) for $s=1,2,\ldots,n-1$, and therefore
\begin{displaymath}
\mathcal{LF}_{2}(\sigma_{1})>\mathcal{LF}_{2}(\sigma_{2})>\ldots>\mathcal{LF}_{2}(\sigma_{n-1})
\end{displaymath}
\item\label{thm:LehmerNormInclusion} The inclusion $\iota_{n}:S_{n}\to S_{n+1}$ given by
\begin{displaymath}
\iota_{n}(\sigma)=(1,\sigma(1)+1,\sigma(2)+1,\ldots,\sigma(n)+1)
\end{displaymath}
preserves $\mathcal{LF}_{2}$.
\item\label{thm:LehmerNormSymmetry} $\mathcal{LF}_{2}(\sigma)=\mathcal{LF}_{2}(\sigma^{-1})$ for all $\sigma \in S_{n}$.  
\item\label{thm:LehmerNormTriangle} $\mathcal{LF}_{2}(\sigma\tau) \leq \mathcal{LF}_{2}(\sigma)+\mathcal{LF}_{2}(\tau)$ for all $\sigma,\tau \in S_{n}$.  
\item\label{thm:LehmerNormTransposition} Let $\tau=\sigma\sigma_{s}$. Then
\begin{displaymath}
\left|\mathcal{LF}_{2}(\tau)-\mathcal{LF}_{2}(\sigma)\right|=2^{-\min\{c_{s}(\sigma),c_{s+1}(\sigma)\}}\mathcal{LF}_{2}(\sigma_{s})
\end{displaymath}
\end{enumerate}
\end{thm}

\begin{proof}{}
(\ref{thm:LehmerNormMinimal}) Note that $\mathcal{LF}_{2}(\sigma)\geq 0$ with the equality holds only if $k_{i}(\sigma)=0$ for $i=0,1,\ldots,n-1$. In such a case $\sigma=e_{n}$.

(\ref{thm:LehmerNormMaximal}) The proof is similar to the one of (\ref{thm:LehmerNormMinimal}). Namely, $\mathcal{LF}_{2}(\sigma)$ is maximal only if $k_{i}(\sigma)=i$ for $i=0,1,\ldots,n-1$ and this implies $\sigma=\bar{e}_{n}$.

(\ref{thm:LehmerNormOrder}) It is enough to see that $c_{i}(\sigma_{s})=\delta_{is}$ 
(see Lemma \ref{lem:LehmerCodeTransposition} (\ref{lem:LehmerCodeTranspositionSingle})).

(\ref{thm:LehmerNormInclusion}) Follows the fact that for $\sigma=(\sigma(1),\sigma(2),\ldots,\sigma(n))$ and 
$\iota_{n}(\sigma)=(1,1+\sigma(1),1+\sigma(2),\ldots,1+\sigma(n))$ we have $c_{1}(\iota_{n}(\sigma))=0$ and 
$c_{i}(\iota_{n}(\sigma))=c_{i-1}(\sigma)$ for $i=2,3,\ldots,n+1$. 

(\ref{thm:LehmerNormSymmetry}) First note, that
\begin{displaymath}
\mathcal{LF}_{2}(\sigma)=\sum_{i=0}^{n-1}\left[2^{i}-2^{i-k_{i}(\sigma)}\right]=
\sum_{i=0}^{n-1}\left[2^{i}-2^{i-c_{n-i}(\sigma)}\right]= 
\sum_{j=1}^{n}\left[2^{n-j}-2^{n-j-c_{j}(\sigma)}\right]
\end{displaymath}
Consequently, the equality 
$\mathcal{LF}_{2}(\sigma)=\mathcal{LF}_{2}(\sigma^{-1})$ is equivalent to
\begin{displaymath}
\sum_{j=1}^{n}2^{n-j-c_{j}(\sigma)}=\sum_{j=1}^{n}2^{n-j-c_{j}(\sigma^{-1})}
\end{displaymath}
Now according to Corollary \ref{cor:LehmerCodeProperties} 
\begin{displaymath}
n-j-c_{j}(\sigma^{-1})=n-j-[c_{\sigma^{-1}(j)}(\sigma)+\sigma^{-1}(j)-j]=n-\sigma^{-1}(j)-c_{\sigma^{-1}(j)}(\sigma)
\end{displaymath}
hence it is enough to notice that
\begin{displaymath}
\sum_{j=1}^{n}2^{n-j-c_{j}(\sigma)}=\sum_{j=1}^{n}2^{n-\sigma^{-1}(j)-c_{\sigma^{-1}(j)}(\sigma)}
\end{displaymath}
is just change of order of summation.
The last equality holds since for $j$ taking all values from $[n]$ the same holds for $\sigma^{-1}(j)$.

(\ref{thm:LehmerNormTriangle}) We have the following equalities:
\begin{displaymath}
\mathcal{LF}_{2}(\sigma)=\sum_{j=1}^{n}\left[2^{n-j}-2^{n-j-c_{j}(\sigma)}\right]
\end{displaymath}
\begin{displaymath}
\mathcal{LF}_{2}(\tau)=\sum_{j=1}^{n}\left[2^{n-j}-2^{n-j-c_{j}(\tau)}\right]
\end{displaymath}
and
\begin{displaymath}
\mathcal{LF}_{2}(\sigma\tau)=\sum_{j=1}^{n}\left[2^{n-j}-2^{n-j-c_{j}(\sigma\tau)}\right]
\end{displaymath}
Therefore the inequality
\begin{displaymath}
\mathcal{LF}_{2}(\sigma\tau) \leq \mathcal{LF}_{2}(\sigma)+\mathcal{LF}_{2}(\tau)
\end{displaymath}
is equivalent to the following ones
\begin{displaymath}
\sum_{j=1}^{n}\left[2^{n-j}-2^{n-j-c_{j}(\sigma\tau)}\right] \leq 
\sum_{j=1}^{n}\left[2^{n-j}-2^{n-j-c_{j}(\sigma)}\right] +
\sum_{j=1}^{n}\left[2^{n-j}-2^{n-j-c_{j}(\tau)}\right]
\end{displaymath}
\begin{displaymath}
\sum_{j=1}^{n}2^{n-j-c_{j}(\tau)}+ 
\sum_{j=1}^{n}2^{n-j-c_{j}(\sigma)} \leq 
\sum_{j=1}^{n}2^{n-j} +
\sum_{j=1}^{n}2^{n-j-c_{j}(\sigma\tau)}
\end{displaymath}
\begin{displaymath}
\sum_{j=1}^{n}\frac{1}{2^{j+c_{j}(\tau)}}+ 
\sum_{j=1}^{n}\frac{1}{2^{j+c_{j}(\sigma)}} \leq 
\sum_{j=1}^{n}\frac{1}{2^{j}} +
\sum_{j=1}^{n}\frac{1}{2^{j+c_{j}(\sigma\tau)}}
\end{displaymath}
\begin{displaymath}
\sum_{j=1}^{n}\frac{1}{2^{j+c_{j}(\tau)}}+ 
\sum_{j=1}^{n}\frac{1}{2^{\tau(j)+c_{\tau(j)}(\sigma)}} \leq 
\sum_{j=1}^{n}\frac{1}{2^{\tau(j)}} +
\sum_{j=1}^{n}\frac{1}{2^{j+c_{j}(\sigma\tau)}}
\end{displaymath}
The last inequality holds since for $j$ taking all values from $[n]$ the same holds for $\tau(j)$.

To finish the proof, it is enough to show that for every $j=1,2,\ldots,n$ we have
\begin{equation}\label{equa:temp}
\frac{1}{2^{j+c_{j}(\tau)}}+\frac{1}{2^{\tau(j)+c_{\tau(j)}(\sigma)}} \leq
\frac{1}{2^{\tau(j)}}+\frac{1}{2^{j+c_{j}(\sigma\tau)}}
\end{equation}
By Lemma \ref{lem:LehmerCodeProperties} (\ref{lem:LehmerCodePropertiesValue}) and (\ref{lem:LehmerCodePropertiesComposition}),
\begin{equation} 
\label{equa:LehmerCodeProperties}
c_{j}(\sigma\tau) \leq c_{j}(\tau)+c_{\tau(j)}(\sigma) \textrm{ and }
\tau(j)\leq j+c_{j}(\tau). 
\end{equation}

Since for non negative numbers $a$, $b$, $c$, $d$ and $e$ satisfying 
\begin{displaymath}
e \leq b+d \textrm{ and } c\leq a+b 
\end{displaymath}
it holds
\begin{displaymath}
\frac{1}{2^{a+b}}+\frac{1}{2^{c+d}}\leq\frac{1}{2^{c}}+\frac{1}{2^{a+e}}
\end{displaymath}
hence (\ref{equa:temp}) is a consequence of (\ref{equa:LehmerCodeProperties}) by substitution
\begin{displaymath}
a = j \textrm{, } b = c_{j}(\tau) \textrm{, } c= \tau(j) \textrm{, } d = c_{\tau(j)}(\sigma) \textrm{ and } e = c_{j}(\sigma\tau) 
\end{displaymath}

(\ref{thm:LehmerNormTransposition})
\begin{gather*}
\mathcal{LF}_{2}(\tau)-\mathcal{LF}_{2}(\sigma)=
\sum_{i=0}^{n-1}\left[2^{i}-2^{i-k_{i}(\tau)}\right]-\sum_{i=0}^{n-1}\left[2^{i}-2^{i-k_{i}(\sigma)}\right]=
\sum_{i=0}^{n-1}\left[2^{i-k_{i}(\sigma)}-2^{i-k_{i}(\tau)}\right]= \\
=\sum_{i=0}^{n-1}\left[2^{i-c_{n-i}(\sigma)}-2^{i-c_{n-i}(\tau)}\right]=
\sum_{j=1}^{n}\left[2^{n-j-c_{j}(\sigma)}-2^{n-j-c_{j}(\tau)}\right]
\end{gather*} 
Now according to Lemma \ref{lem:LehmerCodeTransposition} (\ref{lem:LehmerCodeTranspositionMultiplication}) 
$c_{j}(\tau)=c_{j}(\sigma)$ for $j \neq s,s+1$, hence
\begin{displaymath}
\mathcal{LF}_{2}(\tau)-\mathcal{LF}_{2}(\sigma)=
\left[2^{n-s-c_{s}(\sigma)}-2^{n-s-c_{s}(\tau)}\right]+\left[2^{n-s-1-c_{s+1}(\sigma)}-2^{n-s-1-c_{s+1}(\tau)}\right]
\end{displaymath}
If $\sigma(s)<\sigma(s+1)$, then by Lemma \ref{lem:LehmerCodeTransposition} (\ref{lem:LehmerCodeTranspositionMultiplication})
\begin{displaymath}
c_{s}(\tau)=c_{s+1}(\sigma)+1 \textrm{ and }
c_{s+1}(\tau)=c_{s}(\sigma)
\end{displaymath}
Therefore
\begin{gather*}
\mathcal{LF}_{2}(\tau)-\mathcal{LF}_{2}(\sigma)=
\left[2^{n-s-c_{s}(\sigma)}-2^{n-s-c_{s+1}(\sigma)-1}\right]+\left[2^{n-s-1-c_{s+1}(\sigma)}-2^{n-s-1-c_{s}(\sigma)}\right]=\\
=2^{n-s-1-c_{s}(\sigma)}=2^{-c_{s}(\sigma)}\mathcal{LF}_{2}(\sigma_{s})
\end{gather*}
Otherwise 
\begin{displaymath}
c_{s}(\tau)=c_{s+1}(\sigma) \textrm{ and }
c_{s+1}(\tau)=c_{s}(\sigma)-1 
\end{displaymath}
and therefore
\begin{gather*}
\mathcal{LF}_{2}(\tau)-\mathcal{LF}_{2}(\sigma)=
\left[2^{n-s-c_{s}(\sigma)}-2^{n-s-c_{s+1}(\sigma)}\right]+\left[2^{n-s-1-c_{s+1}(\sigma)}-2^{n-s-1-c_{s}(\sigma)+1}\right]=\\
=-2^{n-s-1-c_{s+1}(\sigma)}=-2^{-c_{s+1}(\sigma)}\mathcal{LF}_{2}(\sigma_{s})
\end{gather*}
The last equalities in both cases are due to (\ref{thm:LehmerNormOrder}).

To finish the proof, it is enough to notice that: if $\sigma(s)<\sigma(s+1)$, then $c_{s}(\sigma) \leq c_{s+1}(s)$, otherwise $c_{s+1}(\sigma)<c_{s}(\sigma)$.
\end{proof}

\section{The distribution of the Lehmer factorial norm}\label{sec:Distribution}

In this section we examine the properties of the probability distribution function of the values of the Lehmer norm. 

We start with the following theorem, the proof of which is due to K. Majcher.

\begin{thm}\label{thm:AllPermutations}
The direct limit of the system of groups $(S_{n},\iota_{n})$ is given by
\begin{displaymath}
\varinjlim{(S_{n},\iota_{n})} \cong S_{\infty}
\end{displaymath}
\end{thm}
\begin{proof}{} Note first, that
\begin{displaymath}
S_{\infty} \cong \varinjlim{(S_{n},j_{n})}
\end{displaymath}
where $j_{n}:S_{n}\to S_{n+1}$ is given by 
\begin{displaymath}
j_{n}(\sigma(1),\sigma(2),\ldots,\sigma(n))=(\sigma(1),\sigma(2),\ldots,\sigma(n),n+1)
\end{displaymath}
To see this for $\sigma \in S_{\infty}$ let $K(\sigma)$ be a minimal natural number such that $\sigma(k)=k$ for all $k\geq K(\sigma)$.
Define 
\begin{displaymath}
F(\sigma) =
\begin{cases}
(1) \in S_{1}  \textrm{ if } K(\sigma)=1 \\ 
(\sigma(1),\ldots,\sigma(K(\sigma)-1)) \in S_{K(\sigma)-1} \textrm{ if } K(\sigma)>2
\end{cases}
\end{displaymath}
(note, that it is impossible to have $K(\sigma)=2$). $F(\sigma)$ determines the unique element
\begin{displaymath} 
G(\sigma)=\left(F(\sigma),j_{K(\sigma)-1}(F(\sigma)),\ldots\right)_{\sim} \in \varinjlim{(S_{n},j_{n})}
\end{displaymath}
It is easy to see that $G$ is the isomorphism between $S_{\infty}$ and $\varinjlim{(S_{n},j_{n})}$.

To finish the proof it is enough to note that for $t_{n}$ being the conjugacy by $\bar{e}_{n}$ the following diagrams commutes:
\begin{displaymath}
\begin{CD}
S_{n} @>\iota_{n}>> S_{n+1} \\
@Vt_{n}VV @VVt_{n+1}V \\
S_{n} @>>j_{n}> S_{n+1} 
\end{CD}
\end{displaymath}
\end{proof}

Note, that due to Theorems \ref{thm:AllPermutations} and \ref{thm:LehmerNorm} (\ref{thm:LehmerNormInclusion}), $\mathcal{LF}_{2}$ 
can be seen as a norm on $S_{\infty}$.

We continue with the following observation concerning the properties of permutations from the image $\iota_{n-1}[S_{n-1}]$. According to the Theorem 
\ref{thm:LehmerNorm} (\ref{thm:LehmerNormMaximal}) and (\ref{thm:LehmerNormInclusion}) we have the following:
\begin{rem}\label{rem:DistributionRestriction}
Let $\sigma \in S_{n}$ be a permutation.
If $c_{1}(\sigma)=0$, then 
\begin{displaymath}
\mathcal{LF}_{2}(\sigma)\leq 2^{n-1}-n
\end{displaymath}
On the other hand, if $c_{1}(\sigma)>0$, then
\begin{displaymath}
\mathcal{LF}_{2}(\sigma)\geq 2^{n-2}
\end{displaymath}
\end{rem} 

The following definition will be crucial to dermine the distribution of $\mathcal{LF}_{2}$ on $S_{\infty}$. 
\begin{defin}\label{def:DecompositionNumbers}
For a natural number $m>0$ and for $k=0,1,\ldots$ let
\begin{gather*}
S_{k}(m)= \bigcup_{t \geq 1}\Big\{ ((m_{1},l_{1}),(m_{2},l_{2}),\ldots,(m_{t},l_{t}))\in \mathbb{N}^{2t}: 
k=m_{1}>m_{2}>\ldots>m_{t}\geq 0 \textrm{; } \\
m_{j}\geq l_{j} \geq 0 \textrm{ for } j=1,2,\ldots,t \textrm{; } 
m=\sum_{j=1}^{t}\left[\sum_{p=0}^{l_{j}} 2^{m_{j}-p}\right]  \Big\}
\end{gather*}
and let $s_{k}(m)=|S_{k}(m)|$.
\end{defin}
\begin{lem}\label{lem:DecompositionNumbers}
Let $m=\sum_{j=1}^{s}2^{m_{j}}$ for some natural numbers $m_{1}>m_{2}>\ldots>m_{s}\geq 0$. Then $s_{k}(m)=0$ for all 
$k \neq m_{1},m_{1}-1$.
\end{lem}
\begin{proof}{} 
Suppose first that $k>m_{1}$. Therefore
\begin{displaymath}
m=\sum_{j=1}^{s}2^{m_{j}}\leq \sum_{i=0}^{m_{1}}2^{i}=2^{m_{1}+1}-1<2^{k}
\end{displaymath}
hence $m$ cannot be decomposed as a sum given by any element 
\begin{displaymath}
\left((m_{1},l_{1}),(m_{2},l_{2}),\ldots,(m_{t},l_{t})\right) \in S_{k}(m)
\end{displaymath}

On the other hand, if $k<m_{1}-1$, then (for $\bar{m}_{1}=k$)
\begin{displaymath}
\sum_{j=1}^{t}\left[\sum_{p=0}^{\bar{l}_{j}} 2^{\bar{m}_{j}-p}\right] \leq
\sum_{i=0}^{k}\left[\sum_{r=0}^{i} 2^{r}\right] =
\sum_{i=0}^{k}\left[2^{i+1}-1\right] = 2^{k+2}-(k+2)<2^{k+2}\leq 2^{m_{1}}\leq m
\end{displaymath} 
and similarily, $m$ cannot be decomposed as a sum given any element 
\begin{displaymath}
\left((\bar{m}_{1},\bar{l}_{1}),(\bar{m}_{2},\bar{l}_{2}),\ldots,(\bar{m}_{t},\bar{l}_{t})\right) \in S_{k}(m)
\end{displaymath}
\end{proof}

According to Lemma \ref{lem:DecompositionNumbers}
\begin{displaymath}
s(m)=\sum_{k=0}^{\infty}s_{k}(m)
\end{displaymath}
is a well defined natural number. Moreover, the following holds:
\begin{thm}\label{thm:Distribution}
Put $s(0)=1$. Then for every natural number $m$ we have
\begin{displaymath} 
s(m) =\left|\left\{\sigma \in S_{\infty}:\mathcal{LF}_{2}(\sigma)=m\right\}\right|
\end{displaymath}
\end{thm}
\begin{proof}{}
The statement is obvious for $m=0$.

Consider $\sigma\in S_{\infty}$ different from the identity. Let $n$ be the minimal natural number such that $\sigma$ can be regarded as an element of $S_{n}$. Since $\sigma$ is not the identity, it follows that $n>1$. Finally, let $m=\mathcal{LF}_{2}(\sigma)$. Therefore
\begin{displaymath}
m=\mathcal{LF}_{2}(\sigma)=\sum_{i=0}^{n-1}\left[2^{i}-2^{i-k_{i}(\sigma)}\right]=
\sum_{\substack{i=0 \\ k_{i}(\sigma)\neq 0}}^{n-1}\left[\sum_{p=0}^{k_{i}(\sigma)-1}2^{(i-1)-p}\right]=
\sum_{\substack{i=1 \\ k_{i}(\sigma)\neq 0}}^{n-1}\left[\sum_{p=0}^{k_{i}(\sigma)-1}2^{(i-1)-p}\right]
\end{displaymath}
Since $n$ is minimal, it follows that $k_{n-1}(\sigma)=c_{1}(\sigma)>0$. 

Let $i_{1}>i_{2}>\ldots>i_{t}$ be all elements of $[n-1]$ such that $k_{i_{j}}(\sigma)\neq 0$ for $j=1,2,\ldots,t$ (of course $t\geq 1$) and let $m_{j}=i_{j}-1$. Thus we have 
\begin{displaymath}
n-2 = m_{1} > m_{2} > \ldots > m_{t} \geq 0
\end{displaymath}
Put $l_{j}=k_{i_{j}}(\sigma)-1$ and note, that $l_{j} \leq m_{j}$.

Since
\begin{displaymath}
m=\sum_{\substack{i=1 \\ k_{i}(\sigma)\neq 0}}^{n-1}\left[\sum_{p=0}^{k_{i}(\sigma)-1}2^{(i-1)-p}\right]=
\sum_{j=1}^{t}\left[\sum_{p=0}^{l_{j}} 2^{m_{j}-p}\right]
\end{displaymath}
this decompostition of $\mathcal{LF}_{2}(\sigma)=m$ determines the element 
\begin{displaymath}
((m_{1},l_{1}),(m_{2},l_{2}),\ldots,(m_{t},l_{t})) \in S_{n-2}(m)
\end{displaymath}

On the other hand, for an element 
\begin{displaymath}
((m_{1},l_{1}),(m_{2},l_{2}),\ldots,(m_{t},l_{t})) \in S_{n}(m)
\end{displaymath} 
let $\sigma \in S_{n+2}$ be given by its Lehmer code in the following way:
\begin{displaymath}
\lehmer(\sigma)=[c_{1}(\sigma),c_{2}(\sigma),\ldots,c_{n+2}(\sigma)]=[k_{n+1}(\sigma),k_{n}(\sigma),\ldots,k_{0}(\sigma)]
\end{displaymath}
where
\begin{itemize}
\item $k_{m_{j}+1}(\sigma)=l_{j}+1$ for $j=1,2,\ldots,t$,
\item $k_{i}(\sigma)=0$ for $i\in \{0,1,\ldots,m+1\}\setminus \{m_{1}+1,m_{2}+1,\ldots,m_{t}+1\}$
\end{itemize} 
For $\sigma$ defined in such a way it holds
\begin{gather*}
\mathcal{LF}_{2}(\sigma)=\sum_{i=0}^{n+1}\left[2^{i}-2^{i-k_{i}(\sigma)}\right]=
\sum_{\substack{i=0 \\ k_{i}(\sigma)\neq 0}}^{n+1}\left[\sum_{p=0}^{k_{i}(\sigma)-1}2^{(i-1)-p}\right]=\\
=\sum_{j=1}^{t}\left[\sum_{p=0}^{k_{m_{j}+1}(\sigma)-1}2^{[(m_{j}+1)-1]-p}\right]=
\sum_{j=1}^{t}\left[\sum_{p=0}^{l_{j}}2^{m_{j}-p}\right]=m
\end{gather*}
Moreover, the assigments 
\begin{displaymath}
\sigma \rightarrow ((m_{1},l_{1}),(m_{2},l_{2}),\ldots,(m_{t},l_{t})) \textrm{ and } 
((m_{1},l_{1}),(m_{2},l_{2}),\ldots,(m_{t},l_{t})) \rightarrow \sigma
\end{displaymath}
are mutually inverse.

This finishes the proof.
\end{proof}

In the next theorem we present the recursive properties of numbers $s_{k}(m)$. Note, that Theorems \ref{thm:Distribution} and \ref{thm:Recursion} together with Lemma \ref{lem:DecompositionNumbers} fully describe the distribution of $\mathcal{LF}_{2}$ on 
$S_{\infty}$. The values of $s_{k}(m)$ for $m=1,2,\ldots,8$, together with the corresponding permutations can be found in Table 
\ref{table:DecompositionNumbers}. The graphs of functions $s(m)$ and $d(m)=\sum_{l=1}^{m}s(l)$ for $m = 1,2,\ldots,256$ are presented in Figure \ref{fig:graphs}. 

\begin{thm}\label{thm:Recursion} The following holds:
\begin{enumerate}[(i)]
\item\label{thm:Recursion0} $s(0)=1$;
\item\label{thm:Recursion1} $s_{0}(1)=1$ and $s_{k}(1)=0$ for $k>0$;
\item\label{thm:Recursion2} $s_{0}(2)=0$, $s_{1}(2)=1$ and $s_{k}(2)=0$ for $k>1$;
\item\label{thm:RecursionPower2} $s_{m}\left(2^{m}\right)=1$,
\begin{displaymath}
s_{m-1}\left(2^{m}\right)=\sum_{j=0}^{m-1}\left[\sum_{k=0}^{m-2}s_{k}\left(2^{m}-\left(2^{m-1}+\ldots+2^{j}\right)\right)\right]
\end{displaymath}
and $s_{k}\left(2^{m}\right)=0$ for $k\neq m,m-1$;
\item\label{thm:RecursionMain1} For $l=0,1,\ldots,m-1$ 
\begin{displaymath}
s_{m}\left(2^{m}+\ldots+2^{l}\right)=1+\sum_{j=l}^{m-1}\left[\sum_{k=0}^{m-1}s_{k}\left(2^{j}+\ldots+2^{l}\right)\right],
\end{displaymath} 
\begin{displaymath}
s_{m-1}\left(2^{m}+\ldots+2^{l}\right)=
\sum_{j=0}^{m-1}\left[\sum_{k=0}^{m-2}s_{k}\left(\left(2^{m}+\ldots+2^{l}\right)-\left(2^{m-1}+\ldots+2^{j}\right)\right)\right]
\end{displaymath} 
and $s_{k}\left(2^{m}+\ldots+2^{l}\right)=0$ for $k\neq m,m-1$;
\item\label{thm:RecursionMain2} For $l=2,3,\ldots,m-1$ and for $a_{0},a_{1},\ldots,a_{l-2}\in \{0,1\}$ not all being equal to $0$
\begin{gather*}
s_{m}\left(2^{m}+\ldots+2^{l}+a_{l-2}2^{l-2}+\ldots+a_{0}2^{0}\right)=\\
\sum_{k=0}^{m-1}\left[s_{k}\left(a_{l-2}2^{l-2}+\ldots+a_{0}2^{0}\right)\right]+
\sum_{j=l}^{m-1}\left[\sum_{k=0}^{m-1}s_{k}\left(2^{j}+\ldots+2^{l}+a_{l-2}2^{l-2}+\ldots+a_{0}2^{0}\right)\right],
\end{gather*}
\begin{gather*}
s_{m-1}\left(2^{m}+\ldots+2^{l}+a_{l-2}2^{l-2}+\ldots+a_{0}2^{0}\right)=\\
=\sum_{j=0}^{m-1}\left[\sum_{k=0}^{m-2}s_{k}\left(\left(2^{m}+\ldots+2^{l}+a_{l-2}2^{l-2}+\ldots+a_{0}2^{0}\right)-
\left(2^{m-1}+\ldots+2^{j}\right)\right)\right]
\end{gather*}
and $s_{k}\left(2^{m}+\ldots+2^{l}+a_{l-2}2^{l-2}+\ldots+a_{0}2^{0}\right)=0$ for $k\neq m,m-1$;
\item\label{thm:RecursionMain3}
For $m>1$ and for $a_{0},a_{1},\ldots,a_{m-2}\in \{0,1\}$ not all being equal to $0$
\begin{displaymath}
s_{m}\left(2^{m}+a_{m-2}2^{m-2}+\ldots+a_{0}2^{0}\right)=
\sum_{k=0}^{m-1}\left[s_{k}\left(a_{m-2}2^{m-2}+\ldots+a_{0}2^{0}\right)\right],
\end{displaymath}
\begin{gather*}
s_{m-1}\left(2^{m}+a_{m-2}2^{m-2}+\ldots+a_{0}2^{0}\right)=\\
=\sum_{j=0}^{m-1}\left[\sum_{k=0}^{m-2}s_{k}\left(\left(2^{m}+a_{m-2}2^{m-2}+\ldots+a_{0}2^{0}\right)-
\left(2^{m-1}+\ldots+2^{j}\right)\right)\right]
\end{gather*}
and $s_{k}\left(2^{m}+a_{m-2}2^{m-2}+\ldots+a_{0}2^{0}\right)=0$ for $k\neq m,m-1$.
\end{enumerate}
\end{thm}
\begin{proof}{}
(\ref{thm:Recursion0}) follows the definition of $s$.

(\ref{thm:Recursion1}) and (\ref{thm:Recursion2}) follows the equations $S_{0}(1)=\{(0,0)\}$ and
$S_{1}(2)=\{(1,0)\}$ respectively, as well as Lemma \ref{lem:DecompositionNumbers}. 

(\ref{thm:RecursionPower2}) There are three cases to consider when determining $S_{k}(2^{m})$, namely $k=m$, $k=m-1$ and $k \neq m,m-1$. 
\begin{enumerate}[(a)]
\item $S_{m}(2^{m})=\{(m,0)\}$, hence $s_{m}(2^{m})=1$.
\item Consider first an element 
\begin{displaymath}
((m_{1},l_{1}),(m_{2},l_{2}),\ldots,(m_{t},l_{t})) \in S_{m-1}(2^{m})
\end{displaymath}
Therefore $m_{1}=m-1$. Putting $j=m-1-l_{1}$ we get $m-1 \geq j \geq 0$. And since 
\begin{displaymath}
2^{m}>\sum_{i=j}^{m-1}2^{i}=\sum_{p=0}^{l_{1}}2^{(m-1)-p}=\sum_{p=0}^{l_{1}}2^{m_{1}-p}
\end{displaymath}
one must have $t>1$. Therefore (note that $m_{2}\leq m-2$) 
\begin{displaymath}
((m_{2},l_{2}),\ldots,(m_{t},l_{t})) \in
\bigcupdot_{k=0}^{m-2}S_{k}\left(2^{m}-\left(\sum_{p=0}^{l_{1}}2^{m_{1}-p}\right)\right)=
\bigcupdot_{k=0}^{m-2}S_{k}\left(2^{m}-\left(2^{m-1}+\ldots+2^{j}\right)\right)
\end{displaymath}

Conversely, for every $j=0,1,\ldots,m-1$ an element
\begin{displaymath}
((m_{2},l_{2}),\ldots,(m_{t},l_{t})) \in \bigcupdot_{k=0}^{m-2}S_{k}\left(2^{m}-\left(2^{m-1}+\ldots+2^{j}\right)\right)
\end{displaymath}
determines 
\begin{displaymath}
((m-1,m-1-j),(m_{2},l_{2}),\ldots,(m_{t},l_{t})) \in S_{m-1}(2^{m})
\end{displaymath}
\item The equality $s_{k}\left(2^{m}\right)=0$ for $k\neq m,m-1$ follows Lemma \ref{lem:DecompositionNumbers}.
\end{enumerate}
(\ref{thm:RecursionMain1}) There are three cases to consider when determining $S_{k}(2^{m}+\ldots+2^{l})$, namely $k=m$, $k=m-1$ and $k \neq m,m-1$. 
\begin{enumerate}[(a)]
\item Consider first an element 
\begin{displaymath}
((m_{1},l_{1}),(m_{2},l_{2}),\ldots,(m_{t},l_{t})) \in S_{m}(2^{m}+\ldots+2^{l})
\end{displaymath}
If $t=1$, then $(m_{1},l_{1})=(m,m-l)$.

Otherwise $l_{1}<m-l$. To see this note that for $l_{1}\geq m-l$ we have
\begin{displaymath} 
[2^{m}+\ldots+2^{m-l_{1}}]+[2^{m_{2}}+\ldots+2^{m_{2}-l_{2}}]\geq [2^{m}+\ldots+2^{l}]+[2^{m_{2}}]>2^{m}+\ldots+2^{l}
\end{displaymath}
hence
\begin{displaymath}
((m_{1},l_{1}),(m_{2},l_{2}),\ldots,(m_{t},l_{t})) \notin S_{m}(2^{m}+\ldots+2^{l})
\end{displaymath}
Let $j=m-1-l_{1}$ (in particular $m-1 \geq j > l-1$). Now 
\begin{displaymath}
2^{m}+\ldots+2^{l}>\sum_{i=j+1}^{m}2^{i}=\sum_{p=0}^{l_{1}}2^{m-p}=\sum_{p=0}^{l_{1}}2^{m_{1}-p}
\end{displaymath}
and therefore (since $m_{2}\leq m-1$)
\begin{displaymath}
((m_{2},l_{2}),\ldots,(m_{t},l_{t})) \in \bigcupdot_{k=0}^{m-1} S_{k}\left(\left(2^{m}+\ldots+2^{l}\right)-
\sum_{p=0}^{l_{1}}2^{m_{1}-p}\right) =
\bigcupdot_{k=0}^{m-1}S_{k}\left(2^{j}+\ldots+2^{l}\right)
\end{displaymath}

Conversely, for every $j=l,\ldots,m-1$ an element
\begin{displaymath}
((m_{2},l_{2}),\ldots,(m_{t},l_{t})) \in \bigcupdot_{k=0}^{m-1}S_{k}\left(2^{j}+\ldots+2^{l}\right)
\end{displaymath} 
defines
\begin{displaymath}
((m,m-1-j),(m_{2},l_{2}),\ldots,(m_{t},l_{t})) \in S_{m}\left(2^{m}+\ldots+2^{l}\right)
\end{displaymath}
Together with $(m,m-l)$ these are all elements of $S_{m}\left(2^{m}+\ldots+2^{l}\right)$.
\item Consider first an element 
\begin{displaymath}
((m_{1},l_{1}),(m_{2},l_{2}),\ldots,(m_{t},l_{t})) \in S_{m-1}(2^{m}+\ldots+2^{l})
\end{displaymath}
Therefore $m_{1}=m-1$. Putting $j=m-1-l_{1}$ we get $m-1 \geq j \geq 0$. Since 
\begin{displaymath}
2^{m}+\ldots+2^{l}>\sum_{i=j}^{m-1}2^{i}=\sum_{p=0}^{l_{1}}2^{(m-1)-p}=\sum_{p=0}^{l_{1}}2^{m_{1}-p}
\end{displaymath}
it follows that $t>1$. Therefore (note that $m_{2}\leq m-2$) 
\begin{gather*}
((m_{2},l_{2}),\ldots,(m_{t},l_{t})) \in
\bigcupdot_{k=0}^{m-2}S_{k}\left(\left(2^{m}+\ldots+2^{l}\right)-\left(\sum_{p=0}^{l_{1}}2^{m_{1}-p}\right)\right)=\\
=\bigcupdot_{k=0}^{m-2}S_{k}\left(\left(2^{m}+\ldots+2^{l}\right)-\left(2^{m-1}+\ldots+2^{j}\right)\right)
\end{gather*}

Conversely, for every $j=0,1,\ldots,m-1$, an element
\begin{displaymath}
((m_{2},l_{2}),\ldots,(m_{t},l_{t})) \in \bigcupdot_{k=0}^{m-2}S_{k}\left(\left(2^{m}+\ldots+2^{l}\right)-\left(2^{m-1}+\ldots+2^{j}\right)\right)
\end{displaymath}
determines 
\begin{displaymath}
((m-1,m-1-j),(m_{2},l_{2}),\ldots,(m_{t},l_{t})) \in S_{m-1}\left(2^{m}+\ldots+2^{l}\right)
\end{displaymath}
\item The equality $s_{k}\left(2^{m}+\ldots+2^{l}\right)=0$ for $k\neq m,m-1$ follows Lemma \ref{lem:DecompositionNumbers}.
\end{enumerate}
(\ref{thm:RecursionMain2}) There are three cases to consider when determining $S_{k}(2^{m}+\ldots+2^{l}+a_{l-2}2^{l-2}+\ldots+a_{0}2^{0})$, namely $k=m$, $k=m-1$ and $k \neq m,m-1$. 
\begin{enumerate}[(a)]
\item Consider first an element 
\begin{displaymath}
((m_{1},l_{1}),(m_{2},l_{2}),\ldots,(m_{t},l_{t})) \in S_{m}(2^{m}+\ldots+2^{l}+a_{l-2}2^{l-2}+\ldots+a_{0}2^{0})
\end{displaymath}
For $l_{1}>m-l$ we have
\begin{displaymath}
2^{m}+\ldots+2^{m-l_{1}} \geq 2^{m}+\ldots+2^{l}+2^{l-1}>2^{m}+\ldots+2^{l}+a_{l-2}2^{l-2}+\ldots+a_{0}2^{2}
\end{displaymath}
and therefore
\begin{displaymath}
((m_{1},l_{1}),(m_{2},l_{2}),\ldots,(m_{t},l_{t})) \notin S_{m}(2^{m}+\ldots+2^{l}+a_{l-2}2^{l-2}+\ldots+a_{0}2^{0})
\end{displaymath}
hence $l_{1}\leq m-l$. And since 
\begin{displaymath}
\sum_{p=0}^{l_{1}}2^{m_{1}-p}\leq \sum_{p=0}^{m-l}2^{m_{1}-p}=\sum_{p=0}^{m-l}2^{m-p}<2^{m}+\ldots+2^{l}+a_{l-2}2^{l-2}+\ldots+a_{0}2^{0}
\end{displaymath}
it follows that $t>1$.

If $l_{1}=m-l$, then
\begin{displaymath} 
\sum_{p=0}^{l_{1}}2^{m_{1}-p}=2^{m}+\ldots+2^{l}
\end{displaymath}
and therefore (since $m_{2}\leq m-1$)
\begin{displaymath}
((m_{2},l_{2}),\ldots,(m_{t},l_{t})) \in \bigcupdot_{k=0}^{m-1} S_{k}\left(a_{l-2}2^{l-2}+\ldots+a_{0}2^{0}\right)
\end{displaymath}

If $l_{1}<m-l$, putting $j=m-1-l_{1}$ we get $l\leq j \leq m-1$. In this case
\begin{displaymath}
\sum_{p=0}^{l_{1}}2^{m_{1}-p}=\sum_{p=0}^{m-1-j}2^{m-p}=2^{m}+\ldots+2^{j+1}
\end{displaymath} 
thus (since $m_{2}\leq m-1$)
\begin{displaymath}
((m_{2},l_{2}),\ldots,(m_{t},l_{t})) \in \bigcupdot_{k=0}^{m-1} S_{k}\left(2^{j}+\ldots+2^{l}+a_{l-2}2^{l-2}+\ldots+a_{0}2^{0}\right)
\end{displaymath}

Conversely, an element
\begin{displaymath}
((m_{2},l_{2}),\ldots,(m_{t},l_{t})) \in \bigcupdot_{k=0}^{m-1} S_{k}\left(a_{l-2}2^{l-2}+\ldots+a_{0}2^{0}\right)
\end{displaymath}
determines
\begin{displaymath}
((m,m-l),(m_{2},l_{2}),\ldots,(m_{t},l_{t})) \in S_{m}(2^{m}+\ldots+2^{l}+a_{l-2}2^{l-2}+\ldots+a_{0}2^{0})
\end{displaymath}
and, for $j=l,\ldots,m-1$, an element
\begin{displaymath}
((m_{2},l_{2}),\ldots,(m_{t},l_{t})) \in \bigcupdot_{k=0}^{m-1} S_{k}\left(2^{j}+\ldots+2^{l}+a_{l-2}2^{l-2}+\ldots+a_{0}2^{0}\right)
\end{displaymath}
determines
\begin{displaymath}
((m,m-1-j),(m_{2},l_{2}),\ldots,(m_{t},l_{t})) \in S_{m}(2^{m}+\ldots+2^{l}+a_{l-2}2^{l-2}+\ldots+a_{0}2^{0})
\end{displaymath}
\item Consider first an element
\begin{displaymath}
((m_{1},l_{1}),(m_{2},l_{2}),\ldots,(m_{t},l_{t})) \in S_{m-1}(2^{m}+\ldots+2^{l}+a_{l-2}2^{l-2}+\ldots+a_{0}2^{0})
\end{displaymath}
Since $m_{1}=m-1$, we get $l_{1}\leq m-1$. Therefore
\begin{displaymath}
\sum_{p=0}^{l_{1}}2^{m_{1}-p}\leq \sum_{p=0}^{m-1}2^{m-1-p}<2^{m}<2^{m}+\ldots+2^{l}+a_{l-2}2^{l-2}+\ldots+a_{0}2^{0}
\end{displaymath}
and hence $t>1$. Putting $j=m_{1}-l_{1}$ we get $0\leq j \leq m-1$. Now
\begin{displaymath}
\sum_{p=0}^{l_{1}}2^{m_{1}-p}=\sum_{p=0}^{m-1-j}2^{m-1-p}=2^{m-1}+\ldots+2^{j}
\end{displaymath}
and therefore (since $m_{2}\leq m-2$)
\begin{displaymath}
((m_{2},l_{2}),\ldots,(m_{t},l_{t})) \in \bigcupdot_{k=1}^{m-2} 
S_{k}\left(\left(2^{m}+\ldots+2^{l}+a_{l-2}2^{l-2}+\ldots+a_{0}2^{0}\right)-\left(2^{m-1}+\ldots+2^{j}\right)\right)
\end{displaymath}
Conversely, for every $j=0,1,\ldots, k-1$ an element
\begin{displaymath}
((m_{2},l_{2}),\ldots,(m_{t},l_{t})) \in \bigcupdot_{k=0}^{m-2} 
S_{k}\left(\left(2^{m}+\ldots+2^{l}+a_{l-2}2^{l-2}+\ldots+a_{0}2^{0}\right)-\left(2^{m-1}+\ldots+2^{j}\right)\right)
\end{displaymath}
determines
\begin{displaymath}
((m-1,m-1-j),(m_{2},l_{2}),\ldots,(m_{t},l_{t})) \in S_{m-1}(2^{m}+\ldots+2^{l}+a_{l-2}2^{l-2}+\ldots+a_{0}2^{0})
\end{displaymath}
\item The equality $s_{k}\left(2^{m}+\ldots+2^{l}+a_{l-2}2^{l-2}+\ldots+a_{0}2^{0}\right)=0$ for $k\neq m,m-1$ follows Lemma \ref{lem:DecompositionNumbers}.
\end{enumerate}
(\ref{thm:RecursionMain3}) There are three cases to consider when determining $S_{k}(2^{m}+a_{m-2}2^{m-2}+\ldots+a_{0}2^{0})$, namely $k=m$, $k=m-1$ and $k \neq m,m-1$.
\begin{enumerate}[(a)]
\item Consider first an element 
\begin{displaymath}
((m_{1},l_{1}),(m_{2},l_{2}),\ldots,(m_{t},l_{t})) \in S_{m}(2^{m}+a_{m-2}2^{m-2}+\ldots+a_{0}2^{0})
\end{displaymath}
For $l_{1}>0$ we have
\begin{displaymath}
\sum_{p=0}^{l_{1}}2^{m_{1}-p}\geq 2^{m}+2^{m-1}>2^{m}+a_{m-2}2^{m-2}+\ldots+a_{0}2^{0}
\end{displaymath}
and therefore
\begin{displaymath}
((m_{1},l_{1}),(m_{2},l_{2}),\ldots,(m_{t},l_{t})) \notin S_{m}(2^{m}+a_{m-2}2^{m-2}+\ldots+a_{0}2^{0})
\end{displaymath}
hence $l_{1}=0$ and $t>1$. Therefore (since $m_{2}<m$)
\begin{displaymath}
((m_{2},l_{2}),\ldots,(m_{t},l_{t})) \in \bigcupdot_{k=0}^{m-1}S_{k}(a_{l-2}2^{l-2}+\ldots+a_{0}2^{0})
\end{displaymath}
Conversely, an element
\begin{displaymath}
((m_{2},l_{2}),\ldots,(m_{t},l_{t})) \in \bigcupdot_{k=0}^{m-1}S_{k}(a_{m-2}2^{m-2}+\ldots+a_{0}2^{0})
\end{displaymath}
determines
\begin{displaymath}
((m,0),(m_{2},l_{2}),\ldots,(m_{t},l_{t})) \in S_{m}(2^{m}+a_{m-2}2^{m-2}+\ldots+a_{0}2^{0})
\end{displaymath}
\item Consider first an element
\begin{displaymath}
((m_{1},l_{1}),(m_{2},l_{2}),\ldots,(m_{t},l_{t})) \in S_{m-1}(2^{m}+a_{m-2}2^{m-2}+\ldots+a_{0}2^{0})
\end{displaymath}
Since $m_{1}=m-1$, we get $l_{1}\leq m-1$. Therefore
\begin{displaymath}
\sum_{p=0}^{l_{1}}2^{m_{1}-p}\leq \sum_{p=0}^{m-1}2^{m-1-p}<2^{m}<2^{m}+a_{m-2}2^{m-2}+\ldots+a_{0}2^{0}
\end{displaymath}
and hence $t>1$. Putting $j=m_{1}-l_{1}$ we get $0\leq j \leq m-1$.
Now 
\begin{displaymath}
\sum_{p=0}^{l_{1}}2^{m_{1}-p}=\sum_{p=0}^{m-1-j}2^{m-1-p}=2^{m-1}+\ldots+2^{j}
\end{displaymath}
and therefore (since $m_{2}\leq m-2$)
\begin{displaymath}
((m_{2},l_{2}),\ldots,(m_{t},l_{t})) \in \bigcupdot_{k=1}^{m-2} 
S_{k}\left(\left(2^{m}+a_{m-2}2^{m-2}+\ldots+a_{0}2^{0}\right)-\left(2^{m-1}+\ldots+2^{j}\right)\right)
\end{displaymath}
Conversely, for every $j=0,1,\ldots,k-1$, an element
\begin{displaymath}
((m_{2},l_{2}),\ldots,(m_{t},l_{t})) \in \bigcupdot_{k=0}^{m-2} 
S_{k}\left(\left(2^{m}+a_{m-2}2^{m-2}+\ldots+a_{0}2^{0}\right)-\left(2^{m-1}+\ldots+2^{j}\right)\right)
\end{displaymath}
determines
\begin{displaymath}
((m-1,m-1-j),(m_{2},l_{2}),\ldots,(m_{t},l_{t})) \in S_{m-1}(2^{m}+a_{m-2}2^{m-2}+\ldots+a_{0}2^{0})
\end{displaymath}
\item The equality $s_{k}\left(2^{m}+a_{m-2}2^{m-2}+\ldots+a_{0}2^{0}\right)=0$ for $k\neq m,m-1$ follows Lemma \ref{lem:DecompositionNumbers}.
\end{enumerate}
This finishes the proof.
\end{proof}

\subsection*{Acknowledgements}
The author is grateful to K. Majcher for the proof of Theorem \ref{thm:AllPermutations}, to P. J\'oziak for the improvement of the proof of Theorem \ref{thm:LehmerNorm} (\ref{thm:LehmerNormTriangle}) as well as for carefully reading of this paper, and to J. Gismatullin for inspiring converations.

Last, but not least, the author wants to thank his whife for her strong and loving support as well as for her inspiring \emph{''try to think nonstandard''}. 

All calculations were performed with R 4.0.3 (\cite{RPackage}).

During the work on this paper the author was partially supported by the SGH fund KAE/S21 and by the NCN fund UMO-2018/31/B/HS4/01005.

\begin{table}[p]
\caption{The Lehmer factorial norm on $S_{3}$}
\label{table:LehmerNorm}
\begin{center}
\begin{tabular}{|c|c|c|c|c|}
\hline
$\numbering_{\lex}(\sigma)$ & $\sigma$ & $\lehmer(\sigma)$ & $\mathcal{LF}_{2}(\sigma)$ & $\numbering_{\lex}(\sigma)$ in the factorial number system representation  \\ 
\hline
$0$ & $(1,2,3)$ & $[0,0,0]$ & $0$ & $0=0\cdot 2! + 0\cdot 1! + 0\cdot 0!$ \\
\hline
$1$ & $(1,3,2)$ & $[0,1,0]$ & $1$ & $1=0\cdot 2! + 1\cdot 1! + 0\cdot 0!$ \\
\hline
$2$ & $(2,1,3)$ & $[1,0,0]$ & $2$ & $2=1\cdot 2! + 0\cdot 1! + 0\cdot 0!$\\
\hline
$3$ & $(2,3,1)$ & $[1,1,0]$ & $3$ & $3=1\cdot 2! + 1\cdot 1! + 0\cdot 0!$ \\
\hline
$4$ & $(3,1,2)$ & $[2,0,0]$ & $3$ & $4=2\cdot 2! + 0\cdot 1! + 0\cdot 0!$\\
\hline
$5$ & $(3,2,1)$ & $[2,1,0]$ & $4$ & $5=2\cdot 2! + 1\cdot 1! + 0\cdot 0!$\\
\hline
\end{tabular}
\end{center}
\end{table}

\begin{table}[p]
\caption{The elements of $S_{k}(m)$}
\label{table:DecompositionNumbers}
\begin{center}
\begin{tabular}{|c|c|c|c|c|}
\hline
$m$ & decomposition of $m$ & $\lehmer(\sigma)$ & $((m_{1},l_{1}),\ldots,(m_{t},l_{t}))$ & element of  \\
\hline
$1$ & $1=[2^{0}]$ & $[1,0]$ & $(0,0)$ & $S_{0}(1)$ \\
\hline
$2$ & $2=[2^{1}]$ & $[1,0,0]$ & $(1,0)$ & $S_{1}(2)$ \\
\hline
$3$ & $3=[2^{1}]+[2^{0}]$ & $[1,1,0]$ & $((1,0),(0,0))$ & $S_{1}(3)$\\
\hline
$3$ & $3=[2^{1}+2^{0}]$ & $[2,0,0]$ & $(1,1)$ & $S_{1}(3)$ \\
\hline
$4$ & $4=[2^{1}+2^{0}]+[2^{0}]$ & $[2,1,0]$ & $((1,1),(0,0))$ & $S_{1}(4)$ \\
\hline
$4$ & $4=[2^{2}]$ & $[1,0,0,0]$ & $(2,0)$ & $S_{2}(4)$\\
\hline
$5$ & $5=[2^{2}] + [2^{0}]$ & $[1,0,1,0]$ & $((2,0),(0,0))$ & $S_{2}(5)$ \\
\hline
$6$ & $6=[2^{2}]+[2^{1}]$ & $[1,1,0,0]$ & $((2,0),(1,0))$ & $S_{2}(6)$\\
\hline
$6$ & $6=[2^{2}+2^{1}]$ & $[2,0,0,0]$ & $(2,1)$ & $S_{2}(6)$ \\
\hline
$7$ & $7=[2^{2}]+[2^{1}]+[2^{0}]$ & $[1,1,1,0]$ & $((2,0),(1,0),(0,0))$ & $S_{2}(7)$ \\
\hline
$7$ & $7=[2^{2}]+[2^{1}+2^{0}]$ & $[1,2,0,0]$ & $((2,0),(1,1))$ & $S_{2}(7)$ \\
\hline
$7$ & $7=[2^{2}+2^{1}]+[2^{0}]$ & $[2,0,1,0]$ & $((2,1),(0,0))$ & $S_{2}(7)$ \\
\hline
$7$ & $7=[2^{2}+2^{1}+2^{0}]$ & $[3,0,0,0]$ & $(2,2)$ & $S_{2}(7)$ \\
\hline
$8$ & $8=[2^{2}]+[2^{1}+2^{0}]+[2^{0}]$ & $[1,2,1,0]$ & $((2,0),(1,1),(0,0))$ & $S_{2}(8)$ \\
\hline
$8$ & $8=[2^{2}+2^{1}]+[2^{1}]$ & $[2,1,0,0]$ & $((2,1),(1,0))$ & $S_{2}(8)$ \\
\hline
$8$ & $8=[2^{2}+2^{1}+2^{0}]+[2^{0}]$ & $[3,0,1,0]$ & $((2,2),(0,0))$ & $S_{2}(8)$ \\
\hline
$8$ & $8=[2^{3}]$ & $[1,0,0,0,0]$ & $(3,0)$ & $S_{3}(8)$ \\
\hline
\end{tabular}
\end{center}
\end{table}

\pgfplotstableread{
k distr cumdistr normcumdistr
1 1 1 0
2 1 2 0
3 2 4 1e-04
4 2 6 1e-04
5 1 7 1e-04
6 2 9 2e-04
7 4 13 3e-04
8 4 17 4e-04
9 4 21 4e-04
10 4 25 5e-04
11 3 28 6e-04
12 3 31 6e-04
13 2 33 7e-04
14 4 37 8e-04
15 8 45 9e-04
16 8 53 0.0011
17 8 61 0.0013
18 10 71 0.0015
19 10 81 0.0017
20 8 89 0.0019
21 10 99 0.0021
22 12 111 0.0023
23 11 122 0.0025
24 11 133 0.0028
25 9 142 0.003
26 6 148 0.0031
27 5 153 0.0032
28 6 159 0.0033
29 4 163 0.0034
30 8 171 0.0036
31 16 187 0.0039
32 16 203 0.0042
33 16 219 0.0046
34 20 239 0.005
35 20 259 0.0054
36 18 277 0.0058
37 22 299 0.0062
38 28 327 0.0068
39 30 357 0.0074
40 28 385 0.008
41 24 409 0.0085
42 24 433 0.009
43 24 457 0.0095
44 22 479 0.01
45 30 509 0.0106
46 38 547 0.0114
47 37 584 0.0122
48 37 621 0.013
49 39 660 0.0138
50 36 696 0.0145
51 34 730 0.0152
52 35 765 0.016
53 32 797 0.0166
54 28 825 0.0172
55 25 850 0.0177
56 21 871 0.0182
57 15 886 0.0185
58 11 897 0.0187
59 10 907 0.0189
60 12 919 0.0192
61 8 927 0.0193
62 16 943 0.0197
63 32 975 0.0203
64 32 1007 0.021
65 32 1039 0.0217
66 40 1079 0.0225
67 40 1119 0.0233
68 36 1155 0.0241
69 44 1199 0.025
70 56 1255 0.0262
71 60 1315 0.0274
72 58 1373 0.0286
73 50 1423 0.0297
74 50 1473 0.0307
75 52 1525 0.0318
76 50 1575 0.0328
77 64 1639 0.0342
78 84 1723 0.0359
79 90 1813 0.0378
80 88 1901 0.0396
81 92 1993 0.0416
82 90 2083 0.0434
83 84 2167 0.0452
84 88 2255 0.047
85 88 2343 0.0489
86 88 2431 0.0507
87 88 2519 0.0525
88 76 2595 0.0541
89 60 2655 0.0554
90 64 2719 0.0567
91 70 2789 0.0582
92 68 2857 0.0596
93 92 2949 0.0615
94 118 3067 0.064
95 119 3186 0.0664
96 119 3305 0.0689
97 125 3430 0.0715
98 126 3556 0.0742
99 128 3684 0.0768
100 135 3819 0.0796
101 144 3963 0.0826
102 148 4111 0.0857
103 138 4249 0.0886
104 124 4373 0.0912
105 128 4501 0.0939
106 127 4628 0.0965
107 125 4753 0.0991
108 140 4893 0.102
109 146 5039 0.1051
110 138 5177 0.108
111 135 5312 0.1108
112 131 5443 0.1135
113 117 5560 0.116
114 108 5668 0.1182
115 102 5770 0.1203
116 85 5855 0.1221
117 68 5923 0.1235
118 57 5980 0.1247
119 46 6026 0.1257
120 38 6064 0.1265
121 29 6093 0.1271
122 22 6115 0.1275
123 20 6135 0.1279
124 24 6159 0.1284
125 16 6175 0.1288
126 32 6207 0.1294
127 64 6271 0.1308
128 64 6335 0.1321
129 64 6399 0.1334
130 80 6479 0.1351
131 80 6559 0.1368
132 72 6631 0.1383
133 88 6719 0.1401
134 112 6831 0.1425
135 120 6951 0.145
136 116 7067 0.1474
137 100 7167 0.1495
138 100 7267 0.1516
139 104 7371 0.1537
140 100 7471 0.1558
141 128 7599 0.1585
142 168 7767 0.162
143 180 7947 0.1657
144 178 8125 0.1694
145 186 8311 0.1733
146 182 8493 0.1771
147 172 8665 0.1807
148 180 8845 0.1845
149 178 9023 0.1882
150 180 9203 0.1919
151 184 9387 0.1958
152 162 9549 0.1991
153 130 9679 0.2019
154 138 9817 0.2047
155 150 9967 0.2079
156 148 10115 0.2109
157 192 10307 0.2149
158 252 10559 0.2202
159 270 10829 0.2258
160 268 11097 0.2314
161 280 11377 0.2373
162 290 11667 0.2433
163 292 11959 0.2494
164 302 12261 0.2557
165 330 12591 0.2626
166 348 12939 0.2698
167 328 13267 0.2767
168 302 13569 0.283
169 302 13871 0.2893
170 300 14171 0.2955
171 304 14475 0.3019
172 336 14811 0.3089
173 360 15171 0.3164
174 368 15539 0.3241
175 376 15915 0.3319
176 364 16279 0.3395
177 340 16619 0.3466
178 324 16943 0.3533
179 304 17247 0.3597
180 282 17529 0.3656
181 254 17783 0.3709
182 246 18029 0.376
183 242 18271 0.381
184 210 18481 0.3854
185 172 18653 0.389
186 192 18845 0.393
187 214 19059 0.3975
188 210 19269 0.4018
189 280 19549 0.4077
190 362 19911 0.4152
191 373 20284 0.423
192 373 20657 0.4308
193 391 21048 0.4389
194 398 21446 0.4472
195 404 21850 0.4557
196 431 22281 0.4647
197 462 22743 0.4743
198 488 23231 0.4845
199 476 23707 0.4944
200 438 24145 0.5035
201 446 24591 0.5128
202 467 25058 0.5226
203 483 25541 0.5326
204 532 26073 0.5437
205 592 26665 0.5561
206 612 27277 0.5689
207 606 27883 0.5815
208 592 28475 0.5938
209 560 29035 0.6055
210 554 29589 0.6171
211 556 30145 0.6287
212 532 30677 0.6398
213 536 31213 0.6509
214 543 31756 0.6623
215 493 32249 0.6725
216 453 32702 0.682
217 488 33190 0.6922
218 511 33701 0.7028
219 532 34233 0.7139
220 604 34837 0.7265
221 652 35489 0.7401
222 650 36139 0.7537
223 645 36784 0.7671
224 641 37425 0.7805
225 643 38068 0.7939
226 646 38714 0.8074
227 644 39358 0.8208
228 641 39999 0.8342
229 628 40627 0.8473
230 584 41211 0.8594
231 558 41769 0.8711
232 556 42325 0.8827
233 536 42861 0.8938
234 525 43386 0.9048
235 527 43913 0.9158
236 500 44413 0.9262
237 458 44871 0.9358
238 431 45302 0.9448
239 396 45698 0.953
240 348 46046 0.9603
241 307 46353 0.9667
242 266 46619 0.9722
243 219 46838 0.9768
244 171 47009 0.9804
245 132 47141 0.9831
246 106 47247 0.9853
247 87 47334 0.9871
248 75 47409 0.9887
249 58 47467 0.9899
250 44 47511 0.9908
251 40 47551 0.9917
252 48 47599 0.9927
253 32 47631 0.9933
254 64 47695 0.9947
255 128 47823 0.9973
256 128 47951 1
}\mytable

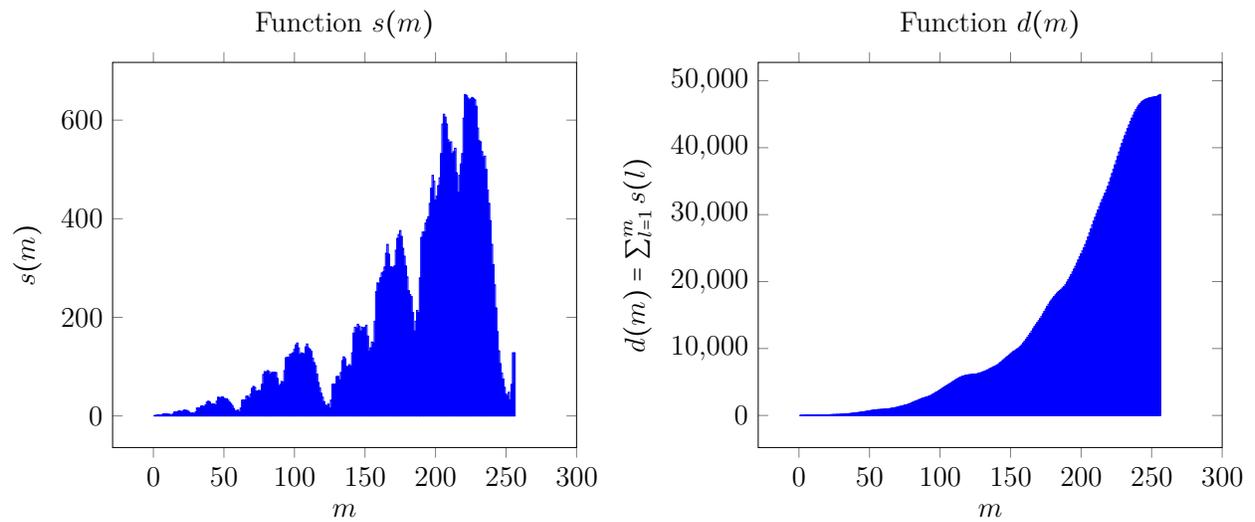
\begin{figure}[p]
\caption{The graphs of functions $s(m)$ and $d(m)$}
\label{fig:graphs}
\begin{tikzpicture}[scale = 0.9]
\begin{axis}[
	ybar, 
	bar width = 0.021cm, 
	xmax = 300, 
	xlabel = {$m$}, 
	ylabel = {$s(m)$}, 
	title = {Function $s(m)$}
]
\addplot table [x=k,y=distr] {\mytable};
\end{axis}
\end{tikzpicture}
\begin{tikzpicture}[scale = 0.9]
\begin{axis}[
	ybar, 
	scaled ticks = false,
	bar width = 0.021cm, 
	xmax = 300, 
	xlabel = {$m$}, 
	ylabel = {$d(m)=\sum_{l=1}^{m}s(l)$}, 
	title = {Function $d(m)$}
]
\addplot table [x=k,y=cumdistr] {\mytable};
\end{axis}
\end{tikzpicture}
\end{figure}

\end{document}